\documentclass[12pt]{article} 

\usepackage{amsmath}
\usepackage{amsthm}
\usepackage{amsfonts}
\usepackage{mathrsfs}
\usepackage{stmaryrd}
\usepackage{setspace}
\usepackage{fullpage}
\usepackage{amssymb}
\usepackage{breqn}
\usepackage{enumitem}
\usepackage{bbold} 
\usepackage{authblk}
\usepackage{comment}
\usepackage{hyperref}
\usepackage{pgf,tikz}
\usepackage{graphicx}
\usepackage{subcaption}

\newtheorem{thm}{Theorem}[section]

\newtheorem{lemma}[thm]{Lemma}

\newtheorem{prop}[thm]{Proposition}

\newtheorem{clm}[thm]{Claim}

\newcommand\ex{\ensuremath{\mathrm{ex}}}

\newcommand\cD{{\mathcal D}}

\newcommand\cH{{\mathcal H}}

\newcommand\cN{{\mathcal N}}

\newcommand{\ignore}[1]{}

\title{On the extremal graphs in generalized Turán problems}
\author{Dániel Gerbner}
\date{\small Alfr\'ed R\'enyi Institute of Mathematics\\
\small \texttt{gerbner.daniel@renyi.hu}}

\begin{document}

\maketitle

\begin{abstract}
Given two graphs $H$ and $F$, the generalized Tur\'an number $\ex(n,H,F)$ is the largest number of copies of $H$ in an $n$-vertex $F$-free graph. For every $F$ and sufficiently large $n$, we present an extremal graph for a generalized Tur\'an problem, i.e., an $F$-free $n$ vertex graph $G$ that for some $H$ contains exactly $\ex(n,H,F)$ copies of $H$.
\end{abstract}

\section{Introduction}

In ordinary Tur\'an problems, we are interested in $\ex(n,F)$, which is the largest number of edges in $n$-vertex graphs that do not contain $F$ as a subgraph. Tur\'an's theorem \cite{T} states that $\ex(n,K_{r+1})=|E(T(n,r))|$, where $T(n,r)$ is the $r$-partite Tur\'an graph, which is the complete $r$-partite graph with each part having order either $\lfloor n/r\rfloor$ or $\lceil n/r\rceil$. The Erd\H os-Stone-Simonovits theorem \cite{ES1966,ES1946} states that for any graph $F$ with $\chi(F)=r+1$ we have $\ex(n,F)=|E(T(n,r))|+o(n^2)$. This gives the asymptotics of $\ex(n,F)$ if $r>1$, but only gives the bound $o(n^2)$ for bipartite $F$. For many bipartite graphs $F$, we do not even know the order of magnitude of $\ex(n,F)$, see \cite{fursim} for a survey.

In generalized Tur\'an problems, we are interested in $\ex(n,H,F)$, which is the largest number of copies of $H$ in $n$-vertex graphs that do not contain $F$ as a subgraph. More formally, let $\cN(H,G)$ denote the number of copies of $H$ in $G$. We let $\ex(n,H,F)=\max\{\cN(H,G): \, \text{$G$ is an $n$-vertex $F$-free graph}\}$. After several sporadic results, the systematic study of generalized Tur\'an problems was initiated by Alon and Shikhelman \cite{ALS2016}.

Given $H$ and $F$, the \textit{extremal graphs} are $n$-vertex $F$-free graphs with $\ex(n,H,F)$ copies of $H$. In this paper, for every graph $F$ with more than one edge, we show a graph $H$ and an extremal $n$-vertex graph $G$ for $\ex(n,H,F)$ if $n$ is sufficiently large. More precisely, we show infinitely many graphs $H$ that do not contain $F$ and for each such $H$, for sufficiently large $n$ we give an $n$-vertex $F$-free graph $G$ with $\ex(n,H,F)=\cN(H,G)$.


Note that in the more precise formulation we avoided the trivial example: in the case $H$ contains $F$, all the $F$-free graphs are extremal.

Let $\chi(F)=r+1$. We denote by $\sigma(F)$ the smallest color class that can appear in an $(r+1)$-coloring of $F$. Given graphs $G$ and $G'$, we denote by $G+G'$ the graph obtained by taking vertex-disjoint copies of $G$ and $G'$ and adding the edges $uv$ for each $u\in V(G), v\in V(G')$. The following theorem constructs an extremal graph for non-bipartite $F$.

\begin{thm}\label{main}
If $F$ has chromatic number $r+1>2$ and $\sigma(F)=s$, $H$ is the complete $r$-partite graph $K_{a,\dots,a}$ with $a$ sufficiently large, then for sufficiently large $n$ we have $\ex(n,H,F)=\cN(H,K_{s-1}+T(n-s+1,r))$.
\end{thm}

Note that the case $s=1$ was proved by the author \cite{ger2}.

Let us assume now that $F$ is bipartite. Let $\overline{K}_{s,t}$ denote the graph we obtain from $K_{s,t}$ by adding all the edges inside the part of order $s$, i.e., $\overline{K}_{s,t}=K_s+T(t,1)$.  If $F$ is not a star, then similar to Theorem \ref{main}, $K_{s-1}+T(n-s+1,1)$ is an extremal graph. Clearly this graph is $F$-free. On the other hand, $F$ is contained in $K_{s,t}$ for $t=|V(F)|-s$. It was shown by Gerbner and Patk\'os \cite{gepat} that $\ex(n,K_{1,q},K_{s,t})=\cN(K_{1,q},K_{s-1}+T(n-s+1,1))$ if $q\ge 2t-1$. This implies that $\ex(n,K_{1,q},F)=\cN(K_{1,q},K_{s-1}+T(n-s+1,1))$.

Finally, if $F$ is a star $K_{1,t}$, let $G$ be the vertex-disjoint union of $\lfloor n/t\rfloor$ copies of $K_t$ and a $K_{t'}$ with $t'=n-t\lfloor n/t\rfloor$. A theorem of Chase \cite{chase} shows that $\ex(n,K_k,K_{1,t})=\cN(K_k,G)$ for any $k\ge 2$, giving us an extremal graph if $t>1$. We also have a simple observation of Cambie, de Verclos and Kang \cite{cvk}, which states that for any tree $T$, $\ex(n,T,K_{1,t})=\cN(T,G)$ for every $n$-vertex graph $G$ that have girth more than $|V(T)|$ and at least $n-1$ vertices of degree $t-1$, where the last vertex has degree $t-1$ or $t-2$. This gives another extremal graph if $t>2$.

\section{Proofs}

We will use a theorem of the author \cite{gerbi}. Given graphs $H$ and $F$ with $\chi(H)<\chi(F)$, we say that $H$ is \textit{weakly $F$-Tur\'an-stable} if the following holds. Any $F$-free $n$-vertex graph $G$ with $\cN(H,G)=\ex(n,H,F)-o(n^{|V(H)|})$ can be turned into a complete $r$-partite graph by adding and removing $o(n^2)$ edges.
Results from \cite{gerb2} imply that $K_{a,\dots,a}$ is 
$F$-Tur\'an-stable. Given a graph $F$ with chromatic number $r+1$, $\cD(F)$ denotes the set of bipartite graphs that appear as two chromatic classes in an $(r+1)$-coloring of $F$.

\begin{thm}\label{thm1}
Let $r+1=\chi(F)>\chi(H)$ and assume that $H$ is weakly $F$-Tur\'an-stable. Then 
for every $n$-vertex $F$-free graph $G$ with $\cN(H,G)=\ex(n,H,F)$ there is an $r$-partition of $V(G)$ to $A_1,\dots,A_r$, a constant $K=K(F)$ and a set $B$ of at most $rK(\sigma(F)-1)$ vertices such that each member of $\cD(F)$ inside the parts shares at least two vertices with $B$, every vertex of $B$ is adjacent to $\Omega(n)$ vertices in each part, every vertex of $A_i\setminus B$ is adjacent to $o(n)$ vertices in $A_i$ and all but $o(n)$ vertices in $A_j$ with $j\neq i$. 
\end{thm}

We will use some results of Gerbner and Patk\'os \cite{gepat}. 

\begin{thm}[Gerbner and Patk\'os \cite{gepat}]\label{gepa}
Let $p+1<s\le t< q$ and $n$ is large enough. 
If $p+q\ge s+t$ or $p=1$ and $q$ is large enough, then $\ex(n,K_{p,q},K_{s,t})=\cN(K_{p,q},\overline{K}_{s-1,n-s+1})$ and every extremal $n$-vertex $K_{s,t}$-free graph contains $\overline{K}_{s-1,n-s+1}$. Moreover, if an $n$-vertex $K_{s,t}$-free graph does not contain $\overline{K}_{s-1,n-s+1}$, then it has at most $ex(n,K_{p,q},K_{s,t})-\Omega(n^{p-1})$ copies of $K_{p,q}$.
\end{thm}

We will also need the following structural information about $K_{s,t}$-free graphs.

\begin{lemma}\label{kiso}
Let $p<s\le t$, $s<q$, $n$ large enough and $G$ be a $K_{s,t}$-free $n$-vertex graph and let $U$ be the set of vertices in $G$ with degree at most $\varepsilon n$. Then the number of copies of $K_{p,q}$ such that the smaller part contains a vertex of $U$ is at most $\varepsilon n^{p+q-1}$. 
\end{lemma}

\begin{proof}
Consider first the vertices of degree at most $\varepsilon' n^{\frac{q-1}{q}}$ for some sufficiently small $\varepsilon'>0$. There are at most $\varepsilon n^{q-1}/2$ ways to choose $q$ neighbors of them, thus there are at most $\varepsilon n^{p+q-1}/2$ copies of $K_{p,q}$ that contains such a vertex in the smaller part.

Let $U'=\{u_1,\dots, u_{|U'|}\}$ denote the vertices in $U$ of degree more than $\varepsilon' n^{\frac{q-1}{q}}$. Consider now the copies of $K_{p,q}$ where the smaller part contains a vertex in $U'$ and call such a vertex the \textit{center} of the $K_{p,q}$. 
Let $S_j$ denote the neighborhood of $u_j$. 
We consider the sets that are the intersections of $s-1$ sets $S_\ell$ with $\ell< j$, then there are $\binom{j-1}{s-1}$ such sets. Each such set shares at most $t-1$ vertices with $S_j$ because of the $K_{s,t}$-free property. 
Therefore, $S_j$ creates at most $(t-1)\binom{j-1}{s-1}$ vertices that are covered at least $s$ times. This implies that after the $j$th set, at most $\sum_{\ell=s}^{j-1}(t-1)\binom{\ell-1}{s-1}\le (t-1)j^s$ vertices are covered at least $s$ times.  At $j=2s n^{\frac{1}{q}}/ \varepsilon'$, at most $c n^{\frac{s}{q}}$ vertices are covered at least $s$ times for some $c$ depending on $s$ and $\varepsilon'$. This shows that $\sum_{\ell=1}^{j} |S_\ell|\le (s-1)n+jc n^{\frac{s}{q}}$. On the other hand, $\sum_{\ell=1}^{j} |S_\ell|\ge j\varepsilon' n^{\frac{q-1}{q}}\ge 2sn$, a contradiction if $n$ is large enough.

We obtained that $|U'|<2s n^{\frac{1}{q}}/ \varepsilon'$. 
Moreover, we have that $\sum_{\ell=1}^{|U'|} |S_\ell|\le (s-1)n+jc n^{\frac{s}{q-1}}\le sn$. 
We will use this above property and that $0\le |S_\ell|\le \varepsilon n$. By the convexity of the binomial function, $\sum_{\ell=1}^j \binom{|S_\ell|}{q}$ is maximal when we have some sets $S_j$ of order $\varepsilon n$ and the others of order 0, i.e., $\sum_{\ell=1}^{|U'|} \binom{|S_\ell|}{q}\le s \binom{\varepsilon n}{q}/\varepsilon\le \varepsilon n^{q}/2$. Now we count the copies of $K_{p,q}$ that contain a vertex of $U'$. First we pick the center and $q$ of its neighbors, at most $\varepsilon n^{q}/2$ ways, then we pick the remaining vertices at most $n^{p-1}$ ways. This completes the proof.
\end{proof}

We will use the following simple observation multiple times. 
\begin{prop}\label{trivv}
Let $V(G)=A_1\cup \dots \cup A_r$ with $|A_i|\ge |V(F)|^2$. Assume that each vertex of $A_i$ is adjacent to all but at most $|A_j|/|V(F)|$ vertices in $A_j$ for each $j\neq i$. Let $F$ be an $(r+1)$-colorable graph such that there is a coloring with a color class of order $s$ and another class of order $t$. Assume that there is a $K_{s,t}$ inside $A_i$. Then there is an $F$ in $G$.
\end{prop}

\begin{proof}
We will embed two color classes into $A_i$, and each other color class greedily into the other sets $A_j$. Each time, when we want to embed some vertices into $A_j$, the already embedded at most $|V(F)|-1$ vertices are adjacent to each but at most $(|V(F)|-1)|A_j|/|V(F)|$ vertices of $A_j$, thus there are enough vertices in their common neighborhood to pick the next color class.
\end{proof}

Now we are ready to prove Theorem \ref{main}. It will be convenient for us to use the $O$ and $o$ notation in most of the proof, but use $\varepsilon$ in some more delicate situation.

\begin{proof} Let $\delta=1/|V(F)|$. We let $a$ be large enough compared to $\delta$. 
We pick $\varepsilon>0$ small enough compared to $\delta,a$ and $F$.
We will apply Theorem \ref{thm1} such a way that every vertex of $A_i\setminus B$ is adjacent to  at most $\varepsilon n$ vertices in $A_i$.

Let  $G$ be an $n$-vertex $F$-free graph with $\cN(H,G)=\ex(n,H,F)$. We apply Theorem \ref{thm1} to obtain vertex sets $A_1,\dots,A_r$ such that $A_1$ is one of the smallest of these sets. It is easy to see that each part has order $\frac{n}{r}+o(n)$. Indeed, otherwise we lose $\Omega(n^{|V(H)|})$ copies of $H$ compared to the Tur\'an graph, and the $o(n^2)$ extra edges inside the parts create $o(n^{|V(H)|})$ copies of $H$. We use this in the form that for each $i$, we have $\frac{n}{r}-\varepsilon n\le |A_i|\le \frac{n}{r}+\varepsilon n$. Let $t$ be the order of the second smallest part of $F$ in an $(r+1)$-coloring of $F$, then $A_i\setminus B$ is $K_{s,t}$-free.

Consider now the different types of copies of $H$. There are at most $\cN(H,T(n,r))$ copies without any edge inside any $A_i$, $O(n^{ra-2})$ copies containing at least two vertices from $B$ and $O(n^{ra-1})$ copies containing one vertex from $B$. 

The remaining copies of $H$ each contain an edge inside some $A_i\setminus B$. Observe that if a copy of $H$ does not contain an edge inside  $A_j\setminus B$, then it contains at most $a$ vertices in  $A_j\setminus B$. Therefore, there is an $i$ such that $H$ contains an edge and at least $a$ vertices in  $A_i\setminus B$.
Observe that $H$ intersects $A_i\setminus B$ in a complete multipartite graph $H'$ on, say, $m\ge a$ vertices. In particular, $H'$ contains a complete bipartite graph $K_{b,m-b}$ with $b\le m-b$ and $b\le a$. On the other hand, $A_i\setminus B$ is $K_{s,t}$-free and as $a\ge t$, we must have $b<s$. Then we have $\ex(n,K_{b,m-b},K_{a,a})=O(n^b)=O(n^{m-1})$ by a result from \cite{gmv}. This implies that for each non-empty subgraph of $H$, there are $O(n^{ra-1})$ copies of $H$ that intersects $A_i\setminus B$ in that subgraph, thus there are $O(n^{ra-1})$ copies of $H$ that contain an edge inside $A_i\setminus B$.

We obtained that $\ex(n,H,F)=\cN(H,T(n,r))+\Theta(n^{ra-1})$. To improve this bound, we will further divide some of the above types of copies of $H$. Let $x$ denote the number of copies of $H$ that intersect $A_i$ in a complete multipartite graph $H'$ on $m\ge a$ vertices with one part being a singleton vertex $v\in A_i\setminus B$. Again, we can just forget about the additional edges and consider the induced subgraph of $H'$ that is $K_{1,m-1}$. By Lemma \ref{kiso}, we have that $x=o(n^{|V(H)|-1})$.

Let $y$ denote the number of copies of $H$ that intersect $A_i$ in a complete multipartite graph $H'$ on $m\ge a$ vertices with one part being a singleton vertex $v\in A_i\cap B$, such that $H'\neq K_{1,a}$ and $H'\neq K_{1,a-1}$. We claim that $y=o(n^{ar-1})$. This statement is obvious if $H'$ contains another vertex of $B$ or an edge inside an $A_j$, since then we can pick $H'$ by picking $v$ and the described vertex or edge, and then any other vertex $n$ ways, to show that there are $o(n^{|V(H')|-1})$ such copies of $H'$. Otherwise, each part of $H'$ without $v$ is inside a part $A_j$, thus the intersection of $H'$ with $A_j\setminus B$ is an independent set on $a$ or $a-1$ vertices. Moreover, the intersection can have $a-1$ vertices for at most one $j$ (since only the center of $K_{1,a}$ can extend such a set to a part of $H$). Therefore, we must have $m\le a+1$, hence $H'=K_{1,a}$ or $K_{1,a-1}$, a contradiction. 


Let $v_1,\dots,v_{|B|}$ be the vertices of $B$ in decreasing order of their degrees. 

\begin{clm}
We have that $|B|\ge s-1$. Moreover, for $j\le s-1$, $v_j$ has at least $(1-\delta)|A_i|$ neighbors in each $A_i$.
\end{clm}

\begin{proof}[Proof of Claim]
We first show that $\sum_{i=1}^s d(v_i)\le sn-|A_1|+\varepsilon  n/|V(F)|$. Indeed, otherwise these $s$ vertices have at least $\varepsilon n/|V(F)|$ common neighbors in each $A_i$. Then we can embed $F$ into $G$ greedily as in Proposition \ref{trivv}. We embed first a color class of $F$ of order $s$ into $v_1,\dots,v_s$ and then we go through the sets $A_i$ one by one to embed the other color classes of $F$. Each time we pick the appropriate number of vertices among the common neighborhood of the already embedded vertices. There are $\varepsilon n/|V(F)|$ common neighbors of $v_1,\dots,v_s$ in $A_i$. The other less than $|V(F)|$ vertices already picked each avoid at most $\varepsilon n/|V(F)|$ neighbors, thus we have at least $\varepsilon n/|V(F)|$ common neighbors. This way we find a copy of $F$ in $G$, a contradiction. This implies that vertices $v_j$ with $j\ge s$ have degree at most $n-|A_1|/s+\varepsilon n/s|V(F)|$, thus they have at most $|A_i|-|A_i|/sr+2\varepsilon n$ neighbors in their part $A_i$. 

Let $z(v)$ denote the number of copies of $H$ that intersect $A_i$ in a $K_{1,a}$ or $K_{1,a-1}$ where $v\in B$ is the center.

By the above part of the proof, there are $o(n^{ar-1})$ other copies of $H$ using at least one edge inside parts. Let $v$ be of degree $n-1$ in $G$, then 
\[z(v)=(1+o(1))\left(\binom{|A_i|-1}{a}\cN(H_0,T(n-|A_i|,r-1))+\binom{|A_i|-1}{a-1}\cN(H_1,T(n-|A_i|,r-1))\right),\] 
where $H_0$ is the complete $(r-1)$-partite graph $K_{a-1,a,\dots,a}$ and $H_1$ is the complete $r$-partite graph $K_{1,a-1,a,\dots,a}$. Let $v'\in B$ a vertex with at most $(1-\delta)|A_i|$ neighbors in $A_i$, then $z(v')\le (1+o(1))(\binom{(1-\delta)|A_i|}{a}\cN(H_0,T((r-1)n/r,r-1))+\binom{(1-\delta)|A_i|}{a-1}\cN(H_1,T((r-1)n/r,r-1))$. If $a$ is large enough compared to $\delta$, then $z(v)/z(v')$ can be arbitrarily large. In particular, if $v'=v_j$ with $j\ge s$, then $z(v)/z(v')\ge rK(s-1)\ge |B|$.

This implies that if $|B|<s-1$, then we have at most $\cN(H,T(n,r))+(s-2)z(v)+|B| z(v')+o(n^{ar-1})<\cN(H,K_{s-1}+T(n-s+1,r))$ copies of $H$, a contradiction. Thus we have
$|B|\ge s-1$. For $j\le s-1$, we have that $v_j$ has more than $(1-\delta)|A_i|$ neighbors in $A_i$, otherwise $\sum_{\ell=1}^{|B|} z(v_\ell)\le (s-2)z(v)+|B|z(v')<(1+o(1))(s-1)z(v)$.
\end{proof}

The above claim implies that there is no $K_{s,t}$ in $B\cup A_i$ using Proposition \ref{trivv}. Moreover, if $|B\cap A_i|=s-s_i$, then there is no $K_{s_i,t}$ in $A_i$ by the same reasoning.

Let $B'=\{v_1,\dots,v_{s-1}\}$.
Let us delete now the edges inside each $A_i\setminus B'$ and add all the edges between parts and all the edges incident to $B'$ to obtain the graph $G'$
, which is isomorphic to $K_{s-1}+T$ for some complete $r$-partite graph $T$.
Let $\cH_0$ denote the set of copies of $H$ in $G$ that contain only edges between parts. Let $\cH_1$ be the set of copies of $H$ in $G$ that intersect an $A_i$ in a star $K_{1,a}$ and contain only edges between parts otherwise. Let $\cH_2$ be the set of copies of $H$ in $G$ that intersect an $A_i$ in a $K_{2,a}$ or $K_{1,1,a}$ or $K_{1,1,a-1}$ or intersect both $A_i$ and $A_j$ in $K_{1,a}$ or $K_{1,a-1}$ for some $i\neq j$ such that the 1-element parts and the 2-element part is in $B$,
and the copy of $H$ contains only edges between parts otherwise. Let $\cH_3$ be the set of other copies of $H$ in $G$. Let $\cH'_i$ denote the set of copies of $H$ in $G'$ defined analogously to $\cH_i$.

Clearly we have $|\cH_0|\le |\cH'_0|$.
By Theorem \ref{gepa}, the number of copies of $K_{1,a}$ in $G$ is at most $\cN(K_{1,a},\overline{K}_{s_i-1,|A_i\cup B|-s_i+1})$. This clearly implies that the number of such copies of $K_{1,a}$ does not decrease when we change $G$ to $G'$. Observe that for each such $K_{1,a}$, the remaining part of $H$ contains only edges between parts and we added all the missing such edges when created $G'$, thus the number of ways to extend such stars also does not decrease, showing that $|\cH_1|<|\cH_1'|$.
We also have $|\cH_2|<|\cH_2'|$
since we can pick the vertices in $B$ (the same in $G$ and $G'$) and then the remaining vertices can be picked from their neighborhoods in $A_i$ or $A_j$. We did not delete any edges incident to vertices of $B$, thus the number of ways to pick such vertices does not decrease.

\begin{clm}
$|\cH_3|\le \varepsilon c n^{ar-2}$ for some constant $c=c(a,r)$. 
\end{clm}

\begin{proof}[Proof of Claim]
The copies of $H$ in $\cH_3$ that contain edges from only one part $A_i$ have the intersection containing $K_{p,a}$ for some $p\le s$. As we deal with $\cH_3$ now, we have that $p\ge 3$.  Clearly there are $O(n^{ar-3})$ such copies of $H$ where each vertex in the part of order $p$ is in $B$. There are at most $\varepsilon n^{p+a-2}$ copies of $K_{p,a}$ inside $A_i$ with a center outside $B$ by Lemma \ref{kiso}. This is multiplied by at most $2^{2a+2p}$ for picking the intersection (a graph containing $K_{p,a})$, by $s$ for picking $p$, by $r$ for picking $i$, and by $n^{ar-a-p}$ for picking the other vertices, thus there are at most $\varepsilon rs n^{ar-2}2^{2a+2p}$ such copies of $H$.

Consider now the copies of $H$ in $\cH_3$ that contain edges inside exactly two parts $A_{i}$ and $A_{j}$. Let $H_i$ denote the intersection of $H$ with $A_i$ and $H_j$ denote the intersection of $H$ with $A_j$. Recall that there are at most $a$ vertices of $H$ in each other part, thus there are at least $2a$ vertices of $H$ in $A_{i} \cup A_j$. As $A_{i}$ and $A_j$ avoid $K_{s,t}$, each of them contains at most $(r-1)(s-1)+a$ vertices of any $H$. Therefore, each of them contains at least $a-(r-1)(s-1)\ge 2t+1$ vertices of $H$. In particular, both $H_i$ and $H_j$ contain a complete bipartite graph with larger part of order more than $t$ (thus smaller part less than $s$). By Theorem \ref{gepa}, there are $O(n^{|V(H_i)|-1})$ copies of $H_i$ inside $A_i$ and $O(n^{|V(H_j)|-1})$ copies of $H_j$ inside $A_j$.

If $H_i$ or $H_j$ contains a smaller part of order more than 1, or contains more than two parts, then there are $O(n^{|V(H_i)|-2})$ copies of $H_i$ inside $A_i$ (or the same holds for $H_j$), thus there are $O(n^{ar-3})$ such copies of $H$.

If $H_i$ and $H_j$ are both stars and the center of $H_i$ is in $A_i\setminus B$, then Lemma \ref{kiso} shows that there are at most $\varepsilon(n^{|V(H_i)|-1})$ copies of $H_i$ in $A_i$. The analogous statement holds for $H_j$, this gives at most $2r(r-1)a^2$ copies of $H$, where we get additional factors for choosing $i$, $j$, which one or both contain a center outside $B$, and how many leaves they contain.

If $H_i$ and $H_j$ are both stars with center in $B$, at least one of them has at most $a-2$ leaves because this copy of $H$ is not in $\cH_2$. Recall that this copy of $H$ contains $r-2$ independent sets in addition to $H_1$ and $H_2$. This is clearly impossible.

Finally, there are at most $O(n^{ar-3})$ copies of $H$ that contain edges in at least three parts.
\end{proof}

Let us return to the proof of the theorem.
Assume that $G[A_i\cup B]$ is not isomorphic to $\overline{K}_{s-1,|A_i\cup B|-s+1}$. Recall that we have shown $|\cH_1|\le |\cH_1'|$. Now we use the moreover part of Theorem \ref{gepa}. It gives that the number of copies of $K_{1,a}$ is at most $\cN(K_{1,a},\overline{K}_{s-1,|A_i\cup B|-s+1})-c_0 |A_i\cup B|^{a-1}$ for some $c_0$ that depends on $a$, $s$ and $t$. This implies that $|\cH_1|\le |\cH_1'|-c_1n^{ar-2}$ for some $c_1$ that depends on $a$, $s$, $t$ and $r$. As $|\cH_0|\le |\cH'_0|$, $|\cH_2|\le |\cH_2'|$ and $|\cH_3| \le c'\varepsilon n^{ar-2}$, we obtain that there are less than $|\cH'_0|+|\cH_1'|+|\cH_2'|$ copies of $H$ in $G$, a contradiction. 

We obtained that each vertex of $B'$ has degree $n-1$. Now assume that a vertex $z\not\in B'$ has at least $t$ neighbors in its part $A_i$ in $G$. Then with the vertices of $B'$ and its neighbors in $A_i$ they form a $K_{s,t}$ inside $A_i$, thus we can use Proposition \ref{trivv} to find an $F$ in $G$, a contradiction.

Assume now that a copy of $H$ in $G$ contains an edge not in $G'$, i.e., an edge inside a part that does not contain a vertex from $B'$. Clearly, $H$ contains at most $a$ vertices from each $A_i\setminus B$ if $H$ does not contain an edge inside $A_i$. Therefore, $H$ contains at least $a$ vertices of $A_j\cup B$ for some $j$ such that $H$ contains an edge inside $A_j\setminus B$. It implies that $H$ contains at least $a-s>r(t+s)$ vertices of $A_j$. Recall that $H$ intersects $A_j\setminus B$ in a complete $r'$-partite graph $H'$ with $2\le r'\le r$. Then there is a part of $H'$ of order more than $t+s$ in it, thus vertices in the other parts of $H'$ each have more than $t$ neighbors inside their part $A_j$, a contradiction.

We obtained that $G$ is contained in $G'$ which is of the form $K_{s-1}+T$ for some complete $r$-partite graph on $n-s+1$ vertices, and the other edges do not create further copies of $H$. It is left to show that $T$ is the Tur\'an graph.
Assume not, i.e., without loss of generality $|A_2\setminus B'|>|A_1\setminus B'|$. Then we move 
$\lfloor(|A_2|-|A_1|)/2\rfloor$ vertices
from $A_2$ to $A_1$ in $G'$, to obtain $G''$. It is well-known and easy to see that the number of edges increases this way. We will use a result of Ma and Qiu \cite{mq} that determines when the $m$-vertex complete bipartite graph with the most copies of $K_{p,q}$ is $T(m,2)$ (an equivalent statement was proved by Brown and Sidorenko \cite{brosid} in a different setting earlier). The result implies that for $a$ large enough, this holds for $K_{a,a}$, $K_{a-1,a}$, $K_{a-1,a-1}$ and $K_{a-2,a}$. For $K_{a,a}$ and $K_{a-1,a}$, this result also appears in a paper \cite{gypl} by Gy\H ori, Pach and Simonovits. 

We claim that the number of copies of $K_{a,a}$ inside $(A_1\cup A_2)\setminus B'$ increases by $\Omega(n^{2a-2})$. Indeed, for each edge, we take $K_{a-1,a-1}$ from the remaining part of $(A_1\cup A_2)\setminus B'$. Let $x$ (resp. $x'$) denote the number of copies of $K_{a-1,a-1}$ inside the remaining part of $(A_1\cup A_2)\setminus B'$ in $G'$ (resp $G''$). Clearly we have $x'\ge x=\Omega(n^{2a-2})$. As each $K_{a,a,}$ is counted $a^2$ ways as the number of edges increases by at least 1, we proved our claim.


We claim that for the following complete $r$-partite graphs: $K=K_{a-1,a,a,\dots, a}$, $K'=K_{a-2,a,a,\dots, a}$ and $K''=K_{a-1,a-1,a,\dots, a}$, the number of copies of them avoiding $B'$ does not decrease. Observe that each copy of $K$, $K'$ and $K''$ intersects $A_1\cup A_2\setminus B'$ in one of $K_{a,a}$, $K_{a-1,a}$, $K_{a-1,a-1}$ and $K_{a-2,a}$. The number of copies of those bipartite graphs does not decrease, and the number of ways to extend them does not change, proving our claim.

The number of copies of $H$ containing exactly one vertex from $B'$ is $(s-1)\cN(K,T)$, and the number of copies of $H$ containing exactly two vertices from $B'$ is $\binom{s-1}{2}(\cN(K',T)+\cN(K'',T))$, thus does not decrease. The number of copies of $H$ containing at least three vertices from $B'$ is $O(n^{ra-3})$, thus $\cN(H,G'')>\cN(H,G')=\cN(H,G)$, a contradiction completing the proof.
\end{proof}




\section{Concluding remarks}

We have shown for any $F$ an extremal graph $G$ in a generalized Tur\'an problem $\ex(n,H,F)$. For any other graphs $H'$, $G$ can be the first candidate to be an extremal graph. However, we show that for each $F\neq K_{1,2}$ there are more than one extremal graphs. 

If $F$ has chromatic number $r+1>2$, one possibility is that for some less balanced complete $r$-partite graph $T$ on $n-s+1$ vertices, $K_{s-1}+T$ contains more copies of $H'$ than $K_{s-1}+T(n-s+1,r)$. It is the case if $H'$ itself is not balanced, e.g., $H=K_{1,a}$ with $a$ large enough. However, $K_{s-1}+T$ is not necessarily extremal either. The author \cite{gerb3} showed a graph $H'$ for every $F$ where complete $r$-partite graphs are not even asymptotically optimal. As $\cN(H',T)=(1+o(1))\cN(H,K_{s-1}+T)$, we obtain that $K_{s-1}+T$ is not extremal for $H'$. If $\cD(F)$ does not contain a forest, then we can add superlinear many edges to a part of the Tur\'an graph without creating $F$, thus $K_{s-1}+T$ is not even extremal for $\ex(n,F)$ (or for any $\ex(n,H',F)$ if $\chi(H')<r$). Finally, $K_{s-1}+T$ does not contain $K_{s+r}$, thus not an extremal graph for $\ex(n,K_{s+r},F)$ (observe that $K_{s+r}$ does not contain $F$, since $F$ has at least $sr$ vertices).

If $F$ is bipartite but not a star, similarly we have that $K_{s-1}+T(n-s+1,1)$ does not contain $K_{s+1}$, which does not contain $F$. Finally, for stars we gave two constructions that are different unless $F=K_{1,2}$, in which case every $F$-free graph is a matching, thus the only extremal graph is a maximal matching.

\smallskip

Gerbner and Palmer \cite{gerpal} gave the following definition. A graph $H$ is $F$-\textit{Tur\'an-good} if $\ex(n,H,F)=\cN(H,T(n,\chi(F)-1))$ for sufficiently large $n$. It was shown by the author \cite{ger2} that for a given $H$, there is an $F$ such that $H$ is $F$-Tur\'an-good if and pnly if $\sigma(H)=1$. In light of our Theorem \ref{main}, it would have been better to introduce a name for graphs with $\ex(n,H,F)=\cN(H,K_{\sigma(F)-1}+T(n-\sigma(F)+1,\chi(F)-1))$ for sufficiently large $n$.

\smallskip

It is a natural question whether we can construct an extremal graph for any $H$ and some $F$. One can immediately see that it is less useful, as there is no reason to expect it to be an extremal graph for any $F'\neq F$. Nevertheless, we can answer this question using a recent result of Morrison, Nir, Norin, Rza{\.z}ewski and Wesolek \cite{mnnrw}. They (resolving a conjecture from \cite{gerpal}) showed that for any graph $H$, if $r$ is large enough, then $H$ is $K_{r+1}$-Tur\'an-good. This gives infinitely many extremal graphs for any $H$.


\bigskip

\textbf{Funding}: Research supported by the National Research, Development and Innovation Office - NKFIH under the grants KH 130371, SNN 129364, FK 132060, and KKP-133819.

\end{document}